\documentclass[11pt,final]{amsart}
\usepackage[active]{srcltx}
\usepackage[latin1]{inputenc}
\usepackage{amsmath,bbm}
\usepackage{amsfonts}
\usepackage{pgf,tikz}
\usetikzlibrary{arrows}
\usepackage{amssymb}
\usepackage[colorlinks=true,urlcolor=blue,
citecolor=red,linkcolor=blue,linktocpage,pdfpagelabels,bookmarksnumbered,bookmarksopen]{hyperref}
\usepackage{graphicx, tikz}
\usepackage{pgf,tikz}
\usetikzlibrary{arrows}
\definecolor{wwqqww}{rgb}{0.4,0,0,4}
\definecolor{wwffqq}{rgb}{0.4,1,0}
\definecolor{fffftt}{rgb}{1,1,0.2}
\definecolor{ccqqcc}{rgb}{0.8,0,0.8}
\definecolor{ffffww}{rgb}{1,1,0.4}
\definecolor{ttfftt}{rgb}{0.2,1,0.2}
\definecolor{fffftt}{rgb}{1,1,0.2}
\definecolor{ccqqww}{rgb}{0.8,0,0.4}
\definecolor{ffffww}{rgb}{1,1,0.4}
\definecolor{ttfftt}{rgb}{0.2,1,0.2}
\definecolor{uququq}{rgb}{0.25,0.25,0.25}
\definecolor{ffqqtt}{rgb}{1,0,0.2}
\definecolor{ttffff}{rgb}{0.2,1,1}
\definecolor{fffftt}{rgb}{1,1,0.2}
\definecolor{wwffqq}{rgb}{0.4,1,0}
\definecolor{qqqqff}{rgb}{0,0,1}
\definecolor{uququq}{rgb}{0.25,0.25,0.25}
\definecolor{ffqqtt}{rgb}{1,0,0.2}
\definecolor{fffftt}{rgb}{1,1,0.2}
\definecolor{wwffqq}{rgb}{0.4,1,0}
\definecolor{qqqqff}{rgb}{0,0,1}
\definecolor{ttfftt}{rgb}{0.2,1,0.2}
\definecolor{qqqqff}{rgb}{0,0,1}
\definecolor{ffqqtt}{rgb}{1,0,0.2}
\definecolor{fffftt}{rgb}{1,1,0.2}
\definecolor{ffffqq}{rgb}{1,1,0}
\definecolor{qqfftt}{rgb}{0,1,0.2}
\tikzset { domaine/.style 2 args={domain=#1:#2} }
\tikzset{
xmin/.store in=\xmin, xmin/.default=-3, xmin=-3,
xmax/.store in=\xmax, xmax/.default=3, xmax=3,
ymin/.store in=\ymin, ymin/.default=-3, ymin=-3,
ymax/.store in=\ymax, ymax/.default=3, ymax=3,
}
\usepackage{pgf,tikz}
\usetikzlibrary{arrows}
\definecolor{ffffqq}{rgb}{1,1,0}
\definecolor{ttfftt}{rgb}{0.2,1,0.2}
\definecolor{qqffcc}{rgb}{0,1,0.8}
\definecolor{ffwwqq}{rgb}{1,0.4,0}
\definecolor{ffqqqq}{rgb}{1,0,0}
\definecolor{ttffqq}{rgb}{0.2,1,0}
\definecolor{qqqqff}{rgb}{0,0,1}
\definecolor{ffqqtt}{rgb}{1,0,0.2}
\definecolor{ffffqq}{rgb}{1,1,0}
\definecolor{qqfftt}{rgb}{0,1,0.2}
\usepackage[english]{babel}
\usepackage{marginnote}
\usepackage[left=2.95cm,right=2.95cm,top=2.8cm,bottom=2.8cm]{geometry}
\numberwithin{equation}{section}
\usepackage{mathrsfs}
\usepackage{color}
\usepackage{amsmath, bbm}
\usepackage{amsfonts}
\usepackage{amssymb}
\usepackage{graphicx}%
\providecommand{\U}[1]{\protect\rule{.1in}{.1in}}
\definecolor{linkcolor}{rgb}{0.00,0.50,0.00}
\providecommand{\U}[1]{\protect\rule{.1in}{.1in}}
\newtheorem{theorem}{Theorem}[section]
\newtheorem{proposition}[theorem]{Proposition}

\newtheorem{definition}{Definition}[section]
\newtheorem{remark}{Remark}[section]

\numberwithin{equation}{section}

\newcommand{\res}{\mathop{\hbox{\vrule height 7pt width .5pt depth 0pt \vrule height .5pt width 6pt depth 0pt}}\nolimits}
\newcommand{\deb}{\rightharpoonup}

\newcommand{\haus}{\mathcal H}

\newcommand{\M}{\mathcal M}

\newcommand{\lcal}{\mathcal L}

\newcommand{\ve}{\varepsilon}

\newcommand{\R}{\mathbb R}

\DeclareMathOperator{\spt}{spt}
\DeclareMathOperator{\Lip}{Lip}

\DeclareMathOperator{\diam}{diam}

\usepackage{pgf,tikz}
\usetikzlibrary{arrows}
\definecolor{zzttqq}{rgb}{0.6,0.2,0}
\definecolor{uququq}{rgb}{0.25,0.25,0.25}
\definecolor{qqqqff}{rgb}{0,0,1}
\definecolor{xdxdff}{rgb}{0.49,0.49,1}
\usepackage{pgf,tikz}
\usetikzlibrary{arrows}
\definecolor{ffffqq}{rgb}{1,1,0}
\definecolor{fftttt}{rgb}{1,0.2,0.2}
\definecolor{ffqqqq}{rgb}{1,0,0}
\definecolor{ffqqtt}{rgb}{1,0,0.2}
\definecolor{ffttqq}{rgb}{1,0.2,0}
\definecolor{xdxdff}{rgb}{0.49,0.49,1}
\definecolor{qqqqff}{rgb}{0,0,1}
\definecolor{ffttww}{rgb}{1,0.2,0.4}
\definecolor{ttqqcc}{rgb}{0.2,0,0.8}
\definecolor{ffqqww}{rgb}{1,0,0.4}
\definecolor{xdxdff}{rgb}{0.49,0.49,1}
\definecolor{qqqqff}{rgb}{0,0,1}
\definecolor{uququq}{rgb}{0.25,0.25,0.25}
\definecolor{ttffww}{rgb}{0.2,1,0.4}
\definecolor{ffqqqq}{rgb}{1,0,0}
\definecolor{qqfftt}{rgb}{0,1,0.2}
\definecolor{zzttqq}{rgb}{0.6,0.2,0}
\definecolor{ffffqq}{rgb}{1,1,0}
\definecolor{qqffww}{rgb}{0,1,0.4}
\definecolor{xdxdff}{rgb}{0.49,0.49,1}
\definecolor{qqqqff}{rgb}{0,0,1}

\definecolor{ffqqqq}{rgb}{1,0,0}
\definecolor{ttqqff}{rgb}{0.2,0,1}
\definecolor{ttqqcc}{rgb}{0.2,0,0.8}
\definecolor{uququq}{rgb}{0.25,0.25,0.25}
\definecolor{xdxdff}{rgb}{0.49,0.49,1}
\definecolor{qqqqff}{rgb}{0,0,1}
\definecolor{ffqqtt}{rgb}{1,0,0.2}

\definecolor{qqffqq}{rgb}{0,1,0}
\definecolor{qqfftt}{rgb}{0,1,0.2}
\definecolor{ffffqq}{rgb}{1,1,0}
\definecolor{ffttcc}{rgb}{1,0.2,0.8}
\definecolor{ffwwqq}{rgb}{1,0.4,0}
\definecolor{ffqqqq}{rgb}{1,0,0}
\definecolor{ffffqq}{rgb}{1,1,0}
\definecolor{qqfftt}{rgb}{0,1,0.2}
\definecolor{ffttcc}{rgb}{1,0.2,0.8}
\definecolor{ffwwqq}{rgb}{1,0.4,0}
\definecolor{ffqqqq}{rgb}{1,0,0}
\usepackage{subfig}
\usetikzlibrary{arrows}
\definecolor{ffqqtt}{rgb}{1,0,0.2}
\definecolor{ffffqq}{rgb}{1,1,0}
\definecolor{ttffqq}{rgb}{0.2,1,0}
\definecolor{wwwwww}{rgb}{0.4,0.4,0.4}
\definecolor{zzqqzz}{rgb}{0.6,0,0,6}
\begin{document}
\title[Summability of transport densities]
{Summability estimates on transport densities\\ with Dirichlet regions on the boundary\\ via symmetrization techniques}
\author[S. Dweik, F. Santambrogio]{Samer Dweik and Filippo Santambrogio}
\address{Laboratoire de Math\'ematiques d'Orsay, Univ. Paris-Sud, CNRS, Universit\'e Paris-Saclay, 91405 Orsay Cedex, France}
\email{samer.dweik@math.u-psud.fr, filippo.santambrogio@math.u-psud.fr}
\maketitle

\begin{abstract} In this paper we consider the mass transportation problem in a bound\-ed
domain $\Omega$ where a positive mass $f^+$ in the interior is sent to the boundary $\partial\Omega$, appearing for instance in some shape optimization problems, and we prove summability estimates on the associated transport density $\sigma$, which is the transport density from a diffuse measure to a measure on the boundary $f^-=P_\#f^+$ ($P$ being the projection on the boundary), hence singular. Via a symmetrization trick, as soon as $\Omega$ is convex or satisfies a uniform exterior ball condition, we prove $L^p$ estimates (if $f^+\in L^p$, then $\sigma\in L^p$). 
%and we also extend these estimates to $L^{p,q}$ Lorentz spaces (the latter being new also in the case without boundary effects). 
Finally, by a counter-example we prove that if $f^+ \in L^{\infty}(\Omega)$ and $f^-$ has bounded density w.r.t. the surface measure on $\partial\Omega$, the transport density $\sigma$ between $f^+$ and $f^-$ is not necessarily in $L^{\infty}(\Omega)$, which means that the fact that $f^-=P_\#f^+$ is crucial.
\end{abstract}

\section{Introduction} 

In optimal transport, the so-called {\it transport density} is an important notion specific to the case of the Monge cost $c(x,y)=|x-y|$, which played different roles in the development of the theory. For instance, in \cite{EvaGan} it was a key object for one of the first proofs of existence of an optimal transport map for such a linear cost. More precisely, such a map was constructed by following integral curves of a vector field $v$ minimizing $\int |v(x)|dx$ under a divergence constraint $\nabla\cdot v=f^+-f^-$, where $f^\pm$ are the two measures to be transported one onto the other, and the transport density $\sigma$ is nothing but $|v|$. More precisely, $\sigma$ and $v$ are intimately connected with the Kantorovich potential $u$: we have $v=-\sigma\nabla u$ and $(\sigma,u)$ solves a particular PDE systems, called Monge-Kantorovich system
\begin{equation}\label{MKsyst1}
\begin{cases}
-\nabla\cdot (\sigma\nabla u)=f=f^+-f^-&\mbox{ in }\Omega\\
\sigma\nabla u\cdot n=0 &\mbox{ on }\partial\Omega\\
|\nabla u|\leq 1&\mbox{ in }\Omega,\\
|\nabla u|= 1&\sigma-\mbox{a.e. }\end{cases}
\end{equation}
The trasnport density $\sigma$ and the optimal vector field $v$ appear in many applications. The role of $\sigma$ has been clarified for its applications to shape optimization problems in \cite{BouButJEMS}. In \cite{Pri,DumIgb,PieCan,DePJim}, the same pair $(\sigma,u)$ also models (in a statical or dynamical framework) the configuration of stable or growing sandpiles, where $u$ gives the pile shape and $\sigma$ stands for sliding layer. The minimal flow problem $\min\{\int|v(x)|dx\,:\,\nabla\cdot v=f\}$, first introduced in \cite{1}, has many strictly convex variants, dating back to \cite{1} itself and later studied in \cite{BraCarSan}, modeling traffic congestion. The minimization of the $L^1$ norm under divergence constraints also has applications in image processing, as in \cite{LelLorSchVal,BriBurGra}, in particular because the $L^1$ norm (and not its strictly convex variants) induces sparsity.

In the framework of both traffic congestion and membrane reinforcement, in \cite{ButOudVel} the authors use a variant of this problem, already present in  \cite{BouButJEMS} and \cite{ButOudSte}, where the density $f$ has not zero mass, but the Monge-Kantorovich system is complemented with Dirichlet boundary condition. In optimal transport terms, this corresponds to the possibility of sending some mass to the boundary. The easiest version of the system becomes 
\begin{equation}\label{MKsyst2}
\begin{cases}
-\nabla\cdot (\sigma\nabla u)=f^+\geq 0&\mbox{ in }\Omega\\
u=0 &\mbox{on }\partial\Omega,\\
|\nabla u|\leq 1&\mbox{ in }\Omega,\\
|\nabla u|= 1&\sigma-\mbox{a.e. }\end{cases}
\end{equation}
and corresponds to an optimal transport problem between $f^+$ and an unknown measure $f^-$, supported on $\partial\Omega$. By optimality, $f^-$ can be proven to be equal to the image of $f^+$ through the projection onto the boundary $\partial\Omega$.

In this paper we are mainly concerned with summability estimates of $\sigma$ in terms of the corresponding summability properties of $f^+$. It is well known, from the works of Feldman-McCann, that $\sigma$ is unique and $L^1$ (i.e., absolutely continuous) as soon as one of the two measures $f^+$ or $f^-$ is absolutely continuous. Then, \cite{6,7,San09} analyzed $L^p$ summability: for $p<d/(d-1)$ ($d$ being the dimension of the ambient space) $\sigma$ is $L^p$ as soon as  one of the two measures $f^+$ or $f^-$ is $L^p$, while for $p\geq d':=d/(d-1)$ (including $p=\infty$) this requires that both are $L^p$. In general, it is not difficult to see that $\sigma$ is never more regular than $f$, and the higher regularity question (i.e. continuity, $C^{0,\alpha}$, $W^{1,p}$\dots) is an open question, matter of current research.

Yet, the summability question in the Dirichlet case is an interesting one, required in some estimates in \cite{ButOudVel}, and it is non-trivial for $p\geq d'$ because $f^-$ is singular. In this paper, thanks to a symmetrization argument, we give positive answer under some geometric conditions on $\partial\Omega$. Note that \cite{5} already contained a similar, but weaker, result: indeed, the methods used in \cite{5} allows to get the $L^p$ estimate we look for, for $p<\infty$, on a convex domain, since a boundary term in an integration by parts happens to have a sign. As far as results are concerned (since, anyway, the strategy is completely different), the novelty in the present paper are the case $p=\infty$ and the case where $\Omega$ only satisfies an exterior ball condition, instead of being convex.

%We also extend all the $L^p$ estimates into $L^{p,q}$ estimates, which appears to be an easy adaptations of the methods in \cite{San09}, but could be useful when looking at the PDE systems \eqref{MKsyst1} and \eqref{MKsyst2} from an harmonic analysis viewpoint. As Dirichlet boundary conditions are often more standard in the PDE community, we decided to carry on the $L^{p,q}$ theory in the Dirichlet case as well, and this explains why we insert these estimates at the end of the present paper. 

The paper is organized as follows. In Section \ref{2} we recall some well known facts, terminology and notations concerning the usual Monge-Kantorovich problem, 
its dual formulation, the role of the transport density, and the variant where a Dirichlet region is inserted. In Section \ref{3}, we will show our main results, namely that the transport density $\sigma$ between $f^+$ and $(P_{\partial\Omega})_{\#} f^+$ is in $L^p(\Omega)$ provided $f^+ \in L^p(\Omega)$, under the assumption that $\Omega$ satisfies an exterior ball condition. This is done via a symmetrization technique which is quite easy to explain in the case where $\Omega$ is a polyhedron: in this case, $\sigma$ is equal to the restriction to $\Omega$ of the transport density from $f^+$ to a new density $f^-$ obtained by symmetrizing $f^+$ across the faces composing the boundary $\partial\Omega$. A similar argument can be performed for domains with ``round'' faces (called {\it round polyhedra}) and, by an approximation argument, for arbitrary domains satisfying an exterior ball condition.  The presentation, for completeness and pedagogical purposes, goes step-by-step from the convex case to the case of domains with an exterior ball condition, by aproximations, and is done for every $p$. At the end of the section we will explain how this could be shortened by directly considering general domains, and how to deduce the result for any $p$ from the case $p=\infty$.
%In Section \ref{5} we switch from $L^p$ spaces to $L^{p,q}$ spaces: we observe that all the estimates on the symmetrized density $f^-$ also work in the $L^{p,q}$ setting, and we prove that the basic estimate (with no Dirichlet regions) on the $L^p$ summability of $\sigma$ (following the proof in \cite{San09}) can be adapted to an $L^{p,q}$ setting. More precisely we prove that if $ f^+,\;f^- \in L^{p,q}(\Omega)$ then $\sigma \in L^{p,q}(\Omega)$ and we also prove the same result for the transport density $\sigma$ between $f^+$ and $(P_{\partial\Omega})_{\#} f^+$, provided that $f^+ \in L^{p,q}(\Omega)$ and $\Omega$ has a uniform exterior ball.
Finally, Section \ref{6} gives an example where $f^+ \in L^{\infty}(\Omega)$ and $f^- \leq C \mathcal{H}^{d-1}\res\partial\Omega$ but the transport density between $f^+$ and $f^-$ is not in $L^{\infty}(\Omega)$. This answer (negatively) to a natural question which could arise reading our results: the uniform ball condition guarantees that $f^+ \in L^{\infty}(\Omega)$ implies that $(P_{\partial\Omega})_{\#} f^+$ has bounded density w.r.t. the Hausdorff measure on $\partial\Omega$ and one could wonder whether this last condition is the good assumption to prove $\sigma\in L^\infty$. But the answer is not.

\section{Preliminaries} \label{2}

\subsection{About the Monge problem, the Beckmann problem, and transport densities}

Given two  finite positive Borel measures $f^+$ and $f^-$ on a compact convex domain $\Omega\subset\R^d$, satisfying the mass balance condition $f^+(\Omega)=f^-(\Omega)$, we consider
$$\Pi(f^+,f^-):=\left\{\gamma \in \mathcal{M}^+( \Omega \times \Omega):\;(\Pi_{x})_{ \#}\gamma =f^+\;,\;(\Pi_{y})_{ \#}\gamma =f^-\;\right\}$$
where $\Pi_{x}$ and $\Pi_y$ are the two projections of $\Omega \times \Omega$ onto $\Omega$. We then consider the minimization  problem
$$\min\left\{\int_{ \Omega \times \Omega}|x-y|\mathrm{d}\gamma:\;\gamma \in \Pi(f^+,f^-)\right\}\qquad(\mbox{KP}).$$
The minimal value of this problem is called $W_1(f^+,f^-)$ and it can be proven to be a distance on the space of fixed-mass measures (usually we use probability measures; if $\Omega$ is unbounded the distance is restricted to the set of measures with finite  first-order moment).

The above problem is the Kantorovich version of the so-called Monge's Problem, which reads
$$\min\left\{\int_{\Omega}|x-T(x)|\mathrm{d}f^+:\;T_\#f^+=f^-\right\}\qquad(\mbox{MP}).$$
For the details about Optimal Transport theory, its history, and the main results, we refer to \cite{villani} or \cite{OTAM}. Even if some of the present considerations are more general than that, for simplicity of the exposition we will assume that $f^+$ is  absolutely continuous w.r.t. the Lebesgue measure. Then, the two problems (KP) and (MP) are equivalent, in the sense that every transport map $T$ such that $T_\#f^+=f^-$ induces a transport plan $\gamma:=(id,T)_\#f^+$ and that, among the optimal $\gamma$ in (KP), there exists one which has this form (on the contrary, there is no uniqueness, and other optimal transport plans could be of different form). The existence of an optimal map $T$ in this problem (or the fact that an optimal $\gamma$ is of the form $(id,T)_\#f^+$) has been a matter of active study between the end of the '90s and the beginning of this century, and we cite in particular \cite{EvaGan,TruWan,CafFelMcC,ambro,ambpra}

In the analysis of the optimal transport problem (KP) above, a key tool consists in convex duality. Indeed, it is possible to prove that the maximization problem below
$$\max\left\{\int_{\Omega}u\,\mathrm{d}(f^+-f^-):\;u\in \Lip_1(\Omega)\right\}\qquad(\mbox{DP})$$
is the dual of (KP): it can be obtained from (KP) by a suitable inf-sup exchange procedure, its value equals $\min$(KP), and for every admissible $\gamma$ in (KP) and every admissible $u$ in (DP) we have
$$\int_{\Omega \times \Omega}|x-y|\mathrm{d}\gamma\geq \int_{\Omega \times \Omega}(u(x)-u(y))\mathrm{d}\gamma=\int_{\Omega}u(x)\mathrm{d}f^+(x)-
\int_{\Omega}u(y)\mathrm{d}f^-(y)=\int_{\Omega}u\mathrm{d}(f^+-f^-).$$
The equality of the two optimal values implies that optimal $\gamma$ and $u$ satisfy $u(x)-u(y)=|x-y|$ on the support of $\gamma$, but also that, whenever we find some admissible $\gamma$ and $u$ satisfying $\int_{\Omega \times \Omega}|x-y|\mathrm{d}\gamma=\int_{\Omega}u\mathrm{d}(f^+-f^-)$, they are both optimal. The maximizers in (DP) are called {\it Kantorovich potentials}.

 In such a theory it is classical to associate with any optimal transport plan $\gamma$ a positive measure $\sigma$ on $\Omega$, called transport density, which represents the amount of transport taking place in each region of $\Omega$. This measure $\sigma$ is defined by 
$$ <\sigma,\varphi>=\int_{\Omega \times \Omega}\mathrm{d}\gamma(x,y)\int_{0}^{1}\varphi(\omega_{x,y}(t))|\dot{\omega}_{x,y}(t)|\mathrm{d}t\;\;\; \mbox{for all}\;\;\varphi\;\in\;C(\Omega)$$
where $\omega_{x,y}$ is a curve parametrizing the straight line segment connecting $x$ to $y$. Notice in particular that one can write 
\begin{equation}\label{sigmaA}
\sigma(A)=\int_{\Omega \times \Omega} \mathcal{H}^1(A \cap [x,y])\mathrm{d}\gamma(x,y)\;\;\;\mbox{for every Borel set}\;A
\end{equation}
where $\mathcal{H}^1$ stands for the 1-dimensional Hausdorff measure. This means that $\sigma(A)$ stands for ``how much'' the transport takes place in $A$, if particles move from their origin $x$ to their destination $y$ on straight lines. We recall some properties of $\sigma$

\begin{proposition}\label{prop transp dens}
 Suppose $f^+\ll\lcal^d$. Then, the transport density $\sigma$ is unique (i.e does not depend on the choice of the optimal transport plan $\gamma$) and $\sigma \ll \mathcal{L}^d$.
Moreover, if both $ f^+, f^- \in L^p(\Omega)$, then $\sigma$ also belongs to $L^p(\Omega)$.\end{proposition}

These properties are well-known in the literature, and we refer to \cite{5}, \cite{6},  \cite{7}, \cite{3} and \cite{San09}, or chapter 4 in \cite{OTAM}. 
The transport density $\sigma$ also arises in the following minimization problem:
$$ \min \left\{\int_{\Omega}|v(x)|\mathrm{d}x:\; v \in L^1(\Omega,\mathbb{R}^d),\;\nabla\cdot v=f^+ -f^-,\; v\cdot n=0 \mbox{ on }\partial\Omega \right\}\qquad \mbox{(BP)}.$$
This is the so-called {\it continuous transportation model} proposed by Beckmann in \cite{1}. Indeed, it is easy to check that the vector field $v$ given by $v=-\sigma\nabla u$ is a solution of the above minimization problem. Also, it is possible to prove (see, for instance, Theorem 4.13 in \cite{OTAM}) that all minimizers are of this form, and that the minimizer is unique as soon as $f^+\ll\lcal^d$. Beware that (BP) should be stated in the space of vector measures, and the divergence condition (with its no-flux boundary condition) should be written in weak form, i.e.
\begin{equation}
\label{BPmeasures}
\min \left\{||v||_\M:\; v \in \M(\Omega,\mathbb{R}^d),\;\int \nabla\phi\cdot \mathrm{d}v+\int \phi  \,\mathrm{d}(f^+ -f^-)=0\;\mbox{ for all }\phi\in C^1(\overline\Omega) \right\}.
\end{equation}
In the above problem, $||v||_\M$ is the norm in the space of measures, but it is possible to prove that the optimizer $v$ is absolutely continuous as soon as $f^+\ll\lcal^d$.

The connection between Beckmann's problem and the optimal transport problem with cost $|x-y|$ can also be seen as a consequence of convex duality. Indeed, if one uses the dual (weak) version of the divergence constraint, we can obtain a dual problem by interchanging inf and sup:
$$\sup\left\{ \int u\,\mathrm{d}(f^+ -f^-)+\inf_v \int (v\cdot \nabla u+|v(x)|)\mathrm{d}x\right\}$$
which becomes
$$\sup\left\{ \int u\,\mathrm{d}(f^+ -f^-)\;:\;|\nabla u|\leq 1\right\}.$$
It is then enough to observe that the condition $|\nabla u|\leq 1$ is equivalent to $u\in \Lip_1$ (here is where we use the convexity of $\Omega$) to get back to (MP) and (KP). 
%\begin{remark} The equivalence between (BP) and (KP)-(MP) is the only point in this section where we need $\Omega$ to be convex.
%\end{remark}

The primal-dual optimality conditions in the above problems can also be written in a PDE form: $\sigma$ solves, together with the Kantorovich potential $u$, the Monge-Kantorovich system \eqref{MKsyst1}.
%
%\begin{equation}\label{MK system}
%\begin{cases}
%-\nabla\cdot (\sigma\nabla u)=f^+-f^-&\mbox{ in }\Omega\\
%|\nabla u|\leq 1&\mbox{ in }\Omega,\\
%|\nabla u|= 1&\sigma-\mbox{a.e. }\end{cases}
%\end{equation}
%
We finish this section with a stability result that we will need in the sequel, and that we express in terms of the minimal flow problem (BP):
\begin{proposition}\label{stabsigma}
Suppose $f^+\ll\lcal^d$ is fixed, $f^-_n\deb f^-$ and suppose that $\spt(f^+),\spt(f^-_n)$ are all contained in a same compact set. Let $\sigma_n$ be the transport density from $f^+$ to $f^-_n$ and $v_n$ the corresponding minimizer of {\rm (BP)}. Then $\sigma_n\deb \sigma$ and $v_n\deb v$, where $\sigma$ is the transport density from $f^+$ to $f^-$ and $v$ is the corresponding minimizer in {\rm (BP)}.
\end{proposition}
\begin{proof}
From $\int|v_n|=W_1(f^+,f^-_n)$ we have a bound on the $L^1$ norm of $v_n$. Hence, up to subsequences, we can assume $v_n\deb \tilde v$ in the sense of measures. The condition $\nabla\cdot v_n=f^+-f^-_n$ passes to the limit, thus giving $\nabla\cdot \tilde v=f^+-f^-$. Moreover, from the semicontinuity of the mass we get $ || \tilde v ||_\M \leq W_1(f^+,f^-)$, which means that $\tilde v$ is optimal in \eqref{BPmeasures} and $ || \tilde v ||_\M= W_1(f^+,f^-)$. Hence $\tilde v=v$ by uniqueness, and we have full convergence of the sequence since the limit does not depend on the subsequence. Concerning $\sigma_n$, we can assume $\sigma_n\deb \tilde\sigma$ (in the sense of measures), and we get $\tilde\sigma\geq |v|$. But the mass passes to the limit, hence $\int d\tilde\sigma=W_1(f^+,f^-)=\int |v|$. This proves $\tilde\sigma=|v|$, which means $\tilde\sigma=\sigma$ and, again, the limit does not depend on the subsequence.
\end{proof}

\subsection{About optimal transport with Dirichlet regions}

In \cite{BouButJEMS} and in \cite{ButOudSte} a transport problem between measures with different mass is proposed, in the presence of a so-called {\it Dirichlet Region}. A Dirichlet region $\Sigma\subset \overline\Omega$ is a closed set where transportation is free, and one can study the following problem
$$
\min\left\{\int_{\overline\Omega \times \overline\Omega }|x-y|\,\mathrm{d}\gamma
   ,\;\gamma \in \Pi_\Sigma(f^+,f^-) \right\},
  $$
where 
$$\Pi_\Sigma(f^+,f^-) :=\left\{\gamma\in \mathcal{M}^+(\overline\Omega \times \overline\Omega)\;:\; ((\Pi_x)_\#\gamma)\res(\overline\Omega\setminus\Sigma)=f^+,\; ((\Pi_y)_\#\gamma)\res(\overline\Omega\setminus\Sigma)=f^-\right\}.$$
It is not difficult to see that this problem corresponds to a transport problem where it is possible to add arbitrary mass to $f^\pm$ on $\Sigma$, but the transport cost between points on $\Sigma$ is set to $0$. A simple variant, that we will not develop here, concerns the case where the mass we add on $\Sigma$ ``pays'' something, i.e. adding a cost $\int g^+(x)\mbox{d}((\Pi_x)_\#\gamma)\res\Sigma+\int g^-(y)\mbox{d}((\Pi_y)_\#\gamma)\res\Sigma$. This is what is done, for instance, in \cite{4} in the case $\Sigma=\partial\Omega$, where $g^\pm$ represent import/export costs.

Anyway, here we consider the easiest case, which is $f^-=0$. In this case the transport plan $\gamma$ can only transport mass from the density $f^+$ on $\overline\Omega$ to $\Sigma$. Since its marginal $(\Pi_y)_\#\gamma$ on $\Sigma$ is completely arbitrary, then it is clear that the optimal choice is to take it equal to $(P_\Sigma)_\#f^+$, where 
$$ P_{\Sigma}(x)=\mbox{argmin}\left\{|x-y|,\;y \in \Sigma\right\} \; \mbox{for all} \;x.$$
By this definition, $P_\Sigma$ is a priori multivalued, but the argmin is a singleton on all the points where the function $x\mapsto d(x,\Sigma)$ is differentiable, which means a.e. (here as well, the assumption $f^+\ll\lcal^d$ is crucial).

In this paper we will concentrate on the case where $\Sigma$ is a negligible (lower-dimensional) subset of $\overline\Omega$. More precisely, for a ``nice'' domain $\Omega$, we will consider $\Sigma=\partial\Omega$ (as in \cite{BouButJEMS,ButOudVel,4}). This means that we will consider the following problem
$$
\min\left\{\int_{\Omega \times \Omega }|x-y|\mathrm{d}\gamma
   ,\;\gamma \in \Pi(f^+,(P_{\partial\Omega})_{\#} f^+) \right\}. 
  $$
  This is also the same as 
 $$ \min\left\{\int_{\Omega \times \Omega }|x-y|\mathrm{d}\gamma
   ,\;(\Pi_x)_\#\gamma=f^+,\; \spt((\Pi_y)_\#\gamma)\subset\partial\Omega \right\}. 
  $$
  In the Beckmann's formulation, this also amounts to solve
  \begin{equation}\label{beckbord}
   \min \left\{\int_{\Omega}|v(x)|\mathrm{d}x:\; v \in L^1(\Omega,\mathbb{R}^d),\;\spt(\nabla\cdot v-f^+ )\subset \partial\Omega\right\}\qquad \mbox{(BP)}.
   \end{equation}
If we write the condition $\spt(\nabla\cdot v-f^+ )\subset \partial\Omega$ as $\nabla\cdot v=f^+$ inside $\overset{\circ}\Omega$, we can express this condition in a weak sense by testing agains functions $u\in C^1_c(\Omega)$ (or $C^1$ functions, vanishing on $\partial \Omega$), and the dual of this problem becomes 
$$\sup\left\{ \int u\,\mathrm{d}(f^+ -f^-)\;:\;u\in C^1(\Omega),\,|\nabla u|\leq 1,\, u=0\,\mbox{ on }\partial\Omega\right\}.$$
This relaxes on the set of $\Lip_1$ functions vanishing on the boundary $\partial \Omega$. In this way, the Dirichlet region $\Sigma$ really hosts a Dirichlet boundary condition!

\begin{remark} Note that a $W_2$ version of this same problem (i.e. the problem 
$$
\min\left\{\int_{\overline\Omega \times \overline\Omega }|x-y|^2\mathrm{d}\gamma
   ,\;\gamma \in \Pi_\Sigma(f^+,f^-) \right\},
  $$
   used to define a distance on $\M(\Omega)$) has been used in \cite{FigGig} in order to study gradient flows with Dirichlet boundary conditions.
  \end{remark}

\begin{remark}  We also observe that in this framework the convexity of $\Omega$ is no longer needed to guarantee the equivalence between {\rm (BP)} and {\rm (KP)}. Indeed, $C^1$ functions vanishing on $\partial \Omega$ {\it are} $\Lip_1$, whatever the shape of $\Omega$! Equivalently, we can think that the transport rays $[x,T(x)]$ will never exit $\Omega$, from the fact that the target measure is on $\partial \Omega$ and is arbitrary: in case of multiple intersections of the segment $[x,T(x)]$ with the boundary, then $P_{\partial\Omega}(x)$ would coincide with the first one.
 \end{remark}

The question that we consider now is whether the transport density $\sigma$ from $f^+$ to $(P_{\partial\Omega})_\#f^+$ (or, equivalently, the optimal vector field $v$ in \eqref{beckbord}) is in  $L^p(\Omega)$ when $f^+ \in L^p(\Omega)$.
  We cannot use Proposition \ref{prop transp dens}, since in this case the target measure 
  $(P_{\partial\Omega})_{\#} f^+$ is concentrated on the boundary of $\Omega$ and hence is not $L^p$ itself. However, we will see that the same $L^p$ result will be true as well, via a technique which will be described in the next section.

 \section{$L^p$ estimates via symmetrization}\label{3}

In this section we will first develop some tools, based on a symmetrization argument, to show that the transport density $\sigma$ from $f^+$ to $(P_{\partial\Omega})_\#f^+$ is also the restriction of a transport density $\tilde\sigma$, which is associated with the transport from $f^+$ to another suitable density $f^-$, supported outside $\Omega$. Then, we will apply this fact so as to produce the desired $L^p$ estimates on $\sigma$.

\begin{figure} 
\begin{tikzpicture}[line cap=round,line join=round,>=triangle 45,x=1.0cm,y=1.0cm]
\clip(-4.22,-2.93) rectangle (5.14,4.27);
\draw [line width=2pt] (1.21,1.88)-- (-2.13,1.88);
\draw [line width=2pt] (-2.13,1.88)-- (-2.93,-1.17);
\draw [line width=2pt] (2.36,0.55)-- (1.21,1.88);
\draw (-2.13,1.88)-- (-1.24,0.72);
\draw [color=ffqqww] (-1.73,2.37)-- (-1.74,1.37);
\draw [color=ffqqww] (-1.51,2.7)-- (-1.54,1.12);
\draw [color=ffqqww] (-1.2,3.06)-- (-1.24,0.72);
\draw (-2.13,1.88)-- (-1.2,3.06);
\draw (1.21,1.88)-- (0.6,3.03);
\draw [dash pattern=on 2pt off 2pt] (-1.2,3.06)-- (-0.74,3.65);
\draw [dash pattern=on 2pt off 2pt] (0.25,3.65)-- (-0.74,3.65);
\draw [dash pattern=on 2pt off 2pt,color=ffqqww] (-0.24,3.65)-- (-0.26,0.15);

\draw [line width=2pt] (-2.89,-1.2)-- (-1.09,-2.5);
\draw [line width=2pt] (-1.09,-2.5)-- (0.91,-2.5);
\draw [line width=2pt] (0.91,-2.5)-- (1.74,-1.99);

%\draw [line width=2pt,dash pattern=on 2pt off 2pt] (-2.89,-1.2)-- (-1.09,-2.5);
%\draw [line width=2pt,dash pattern=on 2pt off 2pt] (-1.09,-2.5)-- (0.91,-2.5);
%\draw [line width=2pt,dash pattern=on 2pt off 2pt] (0.91,-2.5)-- (1.74,-1.99);
\draw [line width=2pt] (2.36,0.55)-- (1.74,-1.99);
\draw [dash pattern=on 2pt off 2pt] (0.6,3.03)-- (0.25,3.65);
\draw (-0.78,0.13)-- (0.22,0.12);
\draw [color=ffqqww] (0.94,1.39)-- (0.94,2.43);
\draw [color=ffqqww] (0.8,2.74)-- (0.79,1.13);
\draw [color=ffqqww] (0.6,3.03)-- (0.58,0.77);
\draw [dash pattern=on 2pt off 2pt,color=ffqqww] (0.23,3.65)-- (0.22,0.12);
\draw [dash pattern=on 2pt off 2pt,color=ffqqww] (-0.76,3.62)-- (-0.78,0.13);
\draw [dash pattern=on 2pt off 2pt] (-0.78,0.13)-- (-1.24,0.72);
\draw [dash pattern=on 2pt off 2pt] (0.22,0.12)-- (0.58,0.77);
\draw (1.21,1.88)-- (0.58,0.77);
\begin{scriptsize}
\fill [color=black] (1.21,1.88) circle (1.5pt);
\fill [color=ttfftt] (-1.24,0.72) circle (1.5pt);
\fill [color=ttfftt] (-1.74,1.37) circle (1.5pt);
\fill [color=ttfftt] (-1.54,1.12) circle (1.5pt);
\fill [color=ffffww] (-1.2,3.06) circle (1.5pt);
\fill [color=ffffww] (-1.51,2.7) circle (1.5pt);
\fill [color=ffffww] (-1.73,2.37) circle (1.5pt);
\fill [color=fffftt] (0.6,3.03) circle (1.5pt);
\fill [color=fffftt] (0.8,2.74) circle (1.5pt);
\fill [color=fffftt] (0.94,2.43) circle (1.5pt);
\fill [color=black] (-2.13,1.88) circle (1.5pt);
\fill [color=wwffqq] (0.94,1.39) circle (1.5pt);
\fill [color=wwffqq] (0.58,0.77) circle (1.5pt);
\fill [color=wwffqq] (0.79,1.13) circle (1.5pt);
\draw[color=black] (0,-0.3) node {{\large$\Omega_i$}};
\draw[color=black] (0,4) node {{\large$R(\Omega_i)$}};
\draw[color=black] (-1,2.2) node {{\large$F_i$}};
\end{scriptsize}
\end{tikzpicture}
\caption{ \label{fig1}} 
\end{figure}
We will start by supposing that $\Omega$ is a convex polyhedron with $n$ faces $F_i$ ($i=1,...,n$), and denote by $\Omega_i$ the set of points whose projection onto $\partial\Omega$ lies in $F_i$:
$$\Omega_i=\left\{x\in\Omega\,:\,d(x,\partial\Omega)=d(x,F_i)\right\}.$$
We can write $\Omega=\bigcup_i\Omega_i$, and the union is almost disjoint (we have $|\Omega_i\cap\Omega_j|=0$ for all $i\neq j$).
 Let $R$ be the map obtained by reflecting with respect to the boundary each subdomain $\Omega_i$. More precisely, for all $x \in \Omega_i\setminus\bigcup_{j\neq i}\Omega_j$, the point $R(x)$ is the reflexion of $x$ with respect to $F_i$ (see Figure \ref{fig1}). In this way $R$ is well-defined for a.e. $x\in\Omega$.
 
Suppose that $f^+ \in L^p(\Omega)$ and set $f^-=R_\# f^+$. It is clear that $f^-$ is an absolutely continous measure, with density given by $f^-(Ry):=f^+(y)\;\;\mbox{for all } y \in \Omega$. Let $\widetilde\Omega $ be any large compact convex set containing $\Omega \cup R(\Omega).$
We observe that $f^- \in L^p(\widetilde\Omega )$
 and $ || f^-||_{L^p}=|| f^+ ||_{L^p}.$ 
 
We are now interested in the following fact concerning the corresponding transport density. We will denote by $\sigma(f^+,f^-)$ the transport density from $f^+$ to $f^-$ (which is unique and belongs to $L^1$ as soon as $f^+\ll\lcal^d$, which will always be the case in our discussion).
  \begin{proposition}\label{convcase}
  Suppose that $\Omega$ is a polyhedron. Take $f^+\ll\lcal^d\res\Omega$ and define $f^-$ as above through $f^-=R_\# f^+$. Then
  $$(\sigma(f^+,f^-))\res\Omega=\sigma(f^+,(P_{\partial\Omega})_\#f^+).$$
Moreover, if $f^+\in L^p(\Omega)$, then the transport density between $f^+$ and $(P_{\partial\Omega})_\#f^+$ is in $L^p(\Omega)$ .
  \end{proposition}
\begin{proof}
First, we will show that $R$ is an optimal transport map from $f^+$ to $f^-$.
Set \begin{equation*}
u(x)=
\begin{cases}
\phantom{-}\mbox{d}(x,\partial\Omega) & \text{if $x \in \Omega$} \\
-\mbox{d}(x,\partial\Omega) & \text{else}
\end{cases}
\end{equation*} 
From $|x-R(x)|=2 |x-P_{\partial\Omega}(x)|,$ we have 
$$\int_{\Omega}|x-R(x)|\mathrm{d}f^+(x)=2 \int_{\Omega} |x-P_{\partial\Omega}(x)|\mathrm{d}f^+(x).$$ 
 On the other hand, $u$ is 1-Lip and
 \begin{eqnarray*}
 \int_{\Omega}u(x)\mathrm{d}(f^+ - f^-)(x)&=&\int_{\Omega} |x-P_{\partial\Omega}(x)|f^+(x)\mathrm{d}x+\int_{\Omega} |R(x)-P_{\partial\Omega}(R(x))|f^+(x)\mathrm{d}x\\
 &=&2 \int_{\Omega} |x-P_{\partial\Omega}(x)|\mathrm{d}f^+(x).
 \end{eqnarray*}
 Consequently, $R$ is an optimal transport map between $f^+$ and $f^-$ and $u$ is a Kantorovich potential.
 
We observe that the segment $[x,R(x)]$ intersects $\partial\Omega$ at the point $P_{\partial\Omega}(x)$ and that we have
  $$[x,R(x)]\cap\Omega=[x,P_{\partial\Omega}(x)].$$
But the map $x\mapsto P_{\partial\Omega}(x)$ is of course optimal in the transport from  $f^+$ to $( P_{\partial\Omega})_\#f^+$. Hence, using \eqref{sigmaA} we immediately get
   $$(\sigma(f^+,f^-))\res\Omega=\sigma(f^+,(P_{\partial\Omega})_\#f^+)$$
and we conclude by using Proposition \ref{prop transp dens}.
   \end{proof}

 Now, we will give a more general construction, inspired from the previous one, which will allow to deal with the case of a domain with an exterior ball condition.

Suppose that the boundary of $\;\Omega\;$ is a union of a finite number of parts of sphere of radius $r$. We will call the domains with these property {\it round polyhedra} (see Figure \ref{Fig. 2}).
 Set again
 $$\;\Omega_i:=\{x\in\;\Omega:\;P_{\partial\Omega}(x) \in F_i\}$$
  where $F_i\subset \partial B(b_i,r)$ is the $i$th part in the boundary of $\Omega$, contained in a sphere centered at $b_i$. More precisely, we suppose that $B:=\bigcup_i B(b_i,r)$ disconnects $\R^d$ and that $\Omega$ is equal to the union of all the bounded connected components of $\R^d\setminus B$.  
  \begin{figure}
\begin{tikzpicture}[line cap=round,line join=round,>=triangle 45,x=1.0cm,y=1.0cm]
\clip(-6.4,-4.2) rectangle (7.06,6.5);
\draw [shift={(-2.81,3.36)}] plot[domain=4.84:5.52,variable=\t]({1*1.93*cos(\t r)+0*1.93*sin(\t r)},{0*1.93*cos(\t r)+1*1.93*sin(\t r)});
\draw [shift={(0,6)}] plot[domain=4.37:5.16,variable=\t]({1*4.2*cos(\t r)+0*4.2*sin(\t r)},{0*4.2*cos(\t r)+1*4.2*sin(\t r)});
\draw [shift={(4.16,3.27)}] plot[domain=3.55:4.66,variable=\t]({1*2.56*cos(\t r)+0*2.56*sin(\t r)},{0*2.56*cos(\t r)+1*2.56*sin(\t r)});
\draw (-1.43,2.05)-- (-1.01,0.29);
\draw [color=ffqqtt] (-0.62,2.49)-- (-1.01,0.29);
\draw [color=ffqqtt] (-0.77,2.39)-- (-1.12,0.76);
\draw [color=ffqqtt] (-0.98,2.29)-- (-1.25,1.3);
\draw [color=ffqqtt] (-1.19,2.17)-- (-1.34,1.69);
\draw [color=ffqqtt] (1.05,2.48)-- (1.49,0.97);
\draw [color=ffqqtt] (1.52,2.31)-- (1.71,1.84);
\draw [shift={(-5.32,-1.08)}] plot[domain=-0.28:0.62,variable=\t]({1*2.33*cos(\t r)+0*2.33*sin(\t r)},{0*2.33*cos(\t r)+1*2.33*sin(\t r)});
\draw [shift={(-4.12,-3.66)},dash pattern=on 3pt off 3pt]  plot[domain=0.11:1.08,variable=\t]({1*2.19*cos(\t r)+0*2.19*sin(\t r)},{0*2.19*cos(\t r)+1*2.19*sin(\t r)});
\draw [shift={(1.07,-7.71)},dash pattern=on 3pt off 3pt]  plot[domain=1.45:2.17,variable=\t]({1*5.22*cos(\t r)+0*5.22*sin(\t r)},{0*5.22*cos(\t r)+1*5.22*sin(\t r)});
\draw [shift={(4.19,-2.91)},dash pattern=on 3pt off 3pt]  plot[domain=2.22:2.99,variable=\t]({1*2.52*cos(\t r)+0*2.52*sin(\t r)},{0*2.52*cos(\t r)+1*2.52*sin(\t r)});
\draw [color=ffqqtt] (1.26,2.4)-- (1.6,1.38);
\draw [dash pattern=on 3pt off 3pt] (-0.84,-0.26)-- (-1.01,0.29);
\draw (1.81,2.25)-- (1.49,0.97);
\draw (1.49,0.97)-- (1.39,0.57);
\draw [dash pattern=on 3pt off 3pt] (1.39,0.57)-- (1.17,-0.31);
\draw [color=ffqqtt] (-0.59,-0.31)-- (-0.31,2.61);
\draw [color=ffqqtt] (-0.21,-0.35)-- (-0.11,2.65);
\draw [color=ffqqtt] (0.18,-0.42)-- (0.1,2.64);
\draw [color=ffqqtt] (1.22,-0.13)-- (0.68,2.61);
\draw [color=ffqqtt] (-0.49,2.56)-- (-0.88,-0.13);
\draw [color=ffqqtt] (0.47,2.65)-- (0.89,-0.33);
\draw [color=ffqqtt] (0.27,2.64)-- (0.51,-0.37);
\draw [color=ffqqtt] (0.9,2.51)-- (1.39,0.57);
\draw [shift={(4.51,-1.07)}] plot[domain=1.84:3.06,variable=\t]({1*1.85*cos(\t r)+0*1.85*sin(\t r)},{0*1.85*cos(\t r)+1*1.85*sin(\t r)});
\draw [shift={(-3.88,1.51)}] plot[domain=5.08:6.24,variable=\t]({1*1.33*cos(\t r)+0*1.33*sin(\t r)},{0*1.33*cos(\t r)+1*1.33*sin(\t r)});
\draw [dotted,color=ffqqtt] (-0.59,-0.31)-- (-0.64,-0.84);
\draw [dotted,color=ffqqtt] (-0.21,-0.33)-- (-0.24,-0.91);
\draw [dotted,color=ffqqtt] (0.89,-0.33)-- (0.96,-0.89);
\draw [dotted,color=ffqqtt] (0.18,-0.42)-- (0.19,-0.8);
\draw [dotted,color=ffqqtt] (0.51,-0.37)-- (0.56,-0.94);
\begin{scriptsize}
\fill [color=black] (-1.43,2.05) circle (1.5pt);
\fill [color=black] (1.81,2.25) circle (1.5pt);
\fill [color=qqfftt] (-1.01,0.29) circle (1.5pt);
\fill [color=qqfftt] (-1.25,1.3) circle (1.5pt);
\fill [color=qqfftt] (-1.34,1.69) circle (1.5pt);
\fill [color=qqfftt] (-1.12,0.76) circle (1.5pt);
\fill [color=qqfftt] (1.49,0.97) circle (1.5pt);
\fill [color=ffffqq] (-0.62,2.49) circle (1.5pt);
\fill [color=ffffqq] (-0.77,2.39) circle (1.5pt);
\fill [color=ffffqq] (-0.98,2.29) circle (1.5pt);
\fill [color=ffffqq] (-1.19,2.17) circle (1.5pt);
\fill [color=fffftt] (1.05,2.48) circle (1.5pt);
\fill [color=ffffqq] (1.52,2.31) circle (1.5pt);
\fill [color=qqqqff] (1.07,-7.71) circle (1.5pt);
\fill [color=ffffqq] (1.26,2.4) circle (1.5pt);
\fill [color=ffqqtt] (0,6) circle (1.5pt);
\draw[color=ffqqtt] (0.28,6.26) node {{\Large$b_i$}};
\fill [color=qqfftt] (1.6,1.38) circle (1.5pt);
\fill [color=qqfftt] (1.39,0.57) circle (1.5pt);
\fill [color=ffffqq] (-0.31,2.61) circle (1.5pt);
\fill [color=ffffqq] (-0.11,2.65) circle (1.5pt);
\fill [color=ffffqq] (0.1,2.64) circle (1.5pt);
\fill [color=ffffqq] (0.68,2.61) circle (1.5pt);
\fill [color=ffffqq] (-0.49,2.56) circle (1.5pt);
\fill [color=ffffqq] (0.27,2.64) circle (1.5pt);
\fill [color=ffffqq] (0.47,2.65) circle (1.5pt);
\fill [color=fffftt] (0.9,2.51) circle (1.5pt);
\fill [color=ttfftt] (1.71,1.84) circle (1.5pt);
\draw[color=black] (0.05,3.3) node {{\Large$T(\Omega_i)$}};
\draw[color=black] (0,-0.04) node {{\Large$\Omega_i$}};
\draw[color=black] (0,1.5) node {{\Large$F_i$}};
\end{scriptsize}
\end{tikzpicture}
\caption{\label{Fig. 2}}
\end{figure}
%{\color{red}   aussi, figure moins symÃÂÃÂÃÂÃÂ©trique. et mettre le $b_i$ du mÃÂÃÂÃÂÃÂªme cÃÂÃÂÃÂÃÂ´tÃÂÃÂÃÂÃÂ©}

 We define
  $$T(x):=b_i+\left(r-\frac{|x-b_i|-r}{L}\;\frac{r}{2}\right)\frac{x-b_i}{|x-b_i|}\;\;\;\mbox{for all}\;\;x \in \Omega_i$$
  where $L:=  \mbox{diam} (\Omega)$ and $b_i$ is the center of the sphere corresponding to $F_i$. Again we choose a large domain $\widetilde\Omega $ containing $\Omega \cup T(\Omega).$
  \begin{proposition} \label{Prop3.2}
Suppose that $f^+ \in L^p(\Omega)$ and  set $f^-:=T_{\#}f^+$, then $ f^- \in L^p(\widetilde\Omega )$ with $|| f^-||_{L^p}\leq C || f^+ ||_{L^p},$ where the constant $C$ only depends on $d,r$ and $L$.
  \end{proposition}
  \begin{proof}
  Compute the Jacobian of the map $T$: on $\Omega_i$, we have
  $$DT(x)=\frac{r}{2L}\left( -I+\frac{r+2L}{|x-b_i|}\big(I-e(x)\otimes e(x)\big)\right),$$
  where $e(x):=(x-b_i)/|x-b_i|$.  It is easy to see that $DT(x)$ is a symmetric matrix with one eigenvalue equal to $-\frac{r}{2L}$ and $d-1$ eigenvalues equal to 
  $$\lambda(x):=\frac{r}{2L}\frac{r+2L-|x-b_i|}{|x-b_i|}=\frac{r}{2L}\left(\frac{r+2L}{|x-b_i|}-1\right).$$
  Using $|x-b_i|\leq r+L$, we get 
  $$ \lambda(x) \geq \frac{r}{2(r+L)}.$$
  This provides, for $J:=|\det(DT)|$, the lower bound 
  $$ J(x)\geq\frac{r^d}{2^d(r+L)^{d-1}L}$$
  which is, by the way, independent of $i$ and of the number of spherical parts composing $\partial\Omega$.
  \\ \\
  From $f^-(T(x))=f^+(x)/J(x)$, we get
  $$\int |f^-(y)|^pdy=\int |f^-(y)|^{p-1}df^-=\int |f^-(T(x))|^{p-1}df^+=\int \frac{f^+(x)^p}{J(x)^{p-1}}dx\leq C\int f^+(x)^pdx,$$
  where $C:=(\inf J(x))^{1-p}$.
  By raising to power $1/p$, this provides 
  $$||f^-||_{L^p}\leq C(r,L,d)^{1/p-1}||f^+||_{L^p}$$
  and the constant can be taken independent of $p$. In particular, the estimate is also valid for $p=\infty$.   \end{proof}
 \begin{proposition}\label{PolyhedraL}
  Suppose that $\Omega$ is a round polyhedron. Take $f^+\ll\lcal^d\res\Omega$ and define $f^-$ as above through $f^-=T_\# f^+$. Then
  $$(\sigma(f^+,f^-))\res\Omega=\sigma(f^+,(P_{\partial\Omega})_\#f^+).$$
Moreover, if $f^+\in L^p(\Omega)$, then the transport density between $f^+$ and $(P_{\partial\Omega})_\#f^+$ is in $L^p(\Omega)$.
  \end{proposition}
  \begin{proof}
  The proof will follow the same lines of Proposition \ref{convcase}.
  We will show again the optimality of $T$ for the transport of $f^+$ to $f^-$ by producing a Kantorovich potential. In this case, we set 
 $$u(x)=\min\limits_{i=1,..,n}|x-b_i|.$$
 The function $u$ is of course 1-Lip and we have 
\begin{eqnarray*}
\int_{\Omega}u(x)\mathrm{d}(f^+ - f^-)(x)&=&\int_{\Omega} u(x)f^+(x)\mathrm{d}x -\int_{\Omega} u(T(x))f^+(x)\mathrm{d}x\\
&=&\sum\limits_{i=1}^{n}\int_{\Omega_i} u(x)f^+(x)\mathrm{d}x -\int_{\Omega_i} u(T(x))f^+(x)\mathrm{d}x\\
&=&\sum\limits_{i=1}^{n}\int_{\Omega_i}\left(|x-b_i|-|T(x)-b_i|\right) f^+(x)\mathrm{d}x.
 \end{eqnarray*}
By definition of $T$, the points $b_i,x$ and $T(x)$ are aligned (with $T(x)\in [x,b_i]$), hence $|x-b_i|-|T(x)-b_i|=|x-T(x)|$ and
$$\int_{\Omega}u(x)\mathrm{d}(f^+ - f^-)(x)=\int_{\Omega}|x-T(x)|f^+(x)\mathrm{d}x.$$
 Consequently, $T$ is the optimal transport map between $f^+$ and $f^-$ and $u$ is the corresponding Kantorovich potential.
 
 %Now that we know that $R$ (or $T$) is an optimal map from $f^+$ to $f^-$,
 Now we observe in this case as well that the segment $[x,T(x)]$ intersects $\partial\Omega$ at the point $P_{\partial\Omega}(x)$ and that we have
  $$ [x,T(x)]\cap\Omega=[x,P_{\partial\Omega}(x)].$$
  %But the map $x\mapsto P_{\partial\Omega}(x)$ is of course optimal in the transport from  $f^+$ to $( P_{\partial\Omega})\#f^+$.
   Hence, using \eqref{sigmaA} again we immediately get
   $$(\sigma(f^+,f^-))\res\Omega=\sigma(f^+,(P_{\partial\Omega})_\#f^+)$$
   and we conclude by Proposition \ref{prop transp dens}.
   \end{proof}
   
   \begin{remark}
   The reader can easily see that, both in Propositions \ref{convcase} and \ref{PolyhedraL}, the restriction property of the transport density $\sigma$ also holds for the vector field $v$.
   \end{remark}
   
    \begin{figure}
\begin{tikzpicture}[line cap=round,line join=round,>=triangle 45,x=1.0cm,y=1.0cm]
\clip(-5.42,-2.7) rectangle (5.41,4.01);
\draw [shift={(0.4,-2.89)},line width=1.2pt,color=ffqqtt]  plot[domain=1.86:2.27,variable=\t]({1*4.94*cos(\t r)+0*4.94*sin(\t r)},{0*4.94*cos(\t r)+1*4.94*sin(\t r)});
\draw [shift={(-0.18,0.58)},line width=1.2pt,color=ffqqtt]  plot[domain=3.02:3.95,variable=\t]({1*2.62*cos(\t r)+0*2.62*sin(\t r)},{0*2.62*cos(\t r)+1*2.62*sin(\t r)});
\draw [shift={(-1.49,2.19)},line width=1.2pt,color=ffqqtt]  plot[domain=4.57:5.33,variable=\t]({1*3.54*cos(\t r)+0*3.54*sin(\t r)},{0*3.54*cos(\t r)+1*3.54*sin(\t r)});
\draw [shift={(1.6,-1.81)},line width=1.2pt,color=ffqqtt]  plot[domain=0.88:2.33,variable=\t]({1*1.53*cos(\t r)+0*1.53*sin(\t r)},{0*1.53*cos(\t r)+1*1.53*sin(\t r)});
\draw [shift={(0.02,0.49)},line width=1.2pt,color=ffqqtt]  plot[domain=-0.42:0.8,variable=\t]({1*2.79*cos(\t r)+0*2.79*sin(\t r)},{0*2.79*cos(\t r)+1*2.79*sin(\t r)});
\draw [shift={(0.18,3.49)},line width=1.2pt,color=ffqqtt]  plot[domain=4.08:5.77,variable=\t]({1*2.03*cos(\t r)+0*2.03*sin(\t r)},{0*2.03*cos(\t r)+1*2.03*sin(\t r)});
\draw [shift={(-1.29,2.26)}] plot[domain=3.82:5.79,variable=\t]({1*0.46*cos(\t r)+0*0.46*sin(\t r)},{0*0.46*cos(\t r)+1*0.46*sin(\t r)});
\draw [shift={(-2.02,1.88)}] plot[domain=-2.3:0.24,variable=\t]({1*0.39*cos(\t r)+0*0.39*sin(\t r)},{0*0.39*cos(\t r)+1*0.39*sin(\t r)});
\draw [shift={(-2.65,1.5)}] plot[domain=-2.25:0.24,variable=\t]({1*0.38*cos(\t r)+0*0.38*sin(\t r)},{0*0.38*cos(\t r)+1*0.38*sin(\t r)});
\draw [shift={(-3.23,0.95)}] plot[domain=-1.17:0.64,variable=\t]({1*0.44*cos(\t r)+0*0.44*sin(\t r)},{0*0.44*cos(\t r)+1*0.44*sin(\t r)});
\draw [shift={(-3.2,0.15)}] plot[domain=-1.05:1.24,variable=\t]({1*0.41*cos(\t r)+0*0.41*sin(\t r)},{0*0.41*cos(\t r)+1*0.41*sin(\t r)});
\draw [shift={(-2.94,-0.58)}] plot[domain=-1.07:1.71,variable=\t]({1*0.38*cos(\t r)+0*0.38*sin(\t r)},{0*0.38*cos(\t r)+1*0.38*sin(\t r)});
\draw [shift={(-2.57,-1.2)}] plot[domain=-0.64:2.14,variable=\t]({1*0.34*cos(\t r)+0*0.34*sin(\t r)},{0*0.34*cos(\t r)+1*0.34*sin(\t r)});
\draw [shift={(-1.3,-1.77)}] plot[domain=0.56:3.01,variable=\t]({1*0.41*cos(\t r)+0*0.41*sin(\t r)},{0*0.41*cos(\t r)+1*0.41*sin(\t r)});
\draw [shift={(-0.49,-1.73)}] plot[domain=0.72:2.77,variable=\t]({1*0.5*cos(\t r)+0*0.5*sin(\t r)},{0*0.5*cos(\t r)+1*0.5*sin(\t r)});
\draw [shift={(0.35,-1.39)}] plot[domain=1.03:3.18,variable=\t]({1*0.47*cos(\t r)+0*0.47*sin(\t r)},{0*0.47*cos(\t r)+1*0.47*sin(\t r)});
\draw [shift={(1.03,-0.9)}] plot[domain=0.96:3.34,variable=\t]({1*0.45*cos(\t r)+0*0.45*sin(\t r)},{0*0.45*cos(\t r)+1*0.45*sin(\t r)});
\draw [shift={(1.63,-0.67)}] plot[domain=0.27:2.77,variable=\t]({1*0.37*cos(\t r)+0*0.37*sin(\t r)},{0*0.37*cos(\t r)+1*0.37*sin(\t r)});
\draw [shift={(3.2,-0.13)}] plot[domain=1.8:3.86,variable=\t]({1*0.45*cos(\t r)+0*0.45*sin(\t r)},{0*0.45*cos(\t r)+1*0.45*sin(\t r)});
\draw [shift={(3.2,0.67)}] plot[domain=2.03:4.46,variable=\t]({1*0.38*cos(\t r)+0*0.38*sin(\t r)},{0*0.38*cos(\t r)+1*0.38*sin(\t r)});
\draw [shift={(3.12,1.42)}] plot[domain=2.12:4.5,variable=\t]({1*0.42*cos(\t r)+0*0.42*sin(\t r)},{0*0.42*cos(\t r)+1*0.42*sin(\t r)});
\draw [shift={(0.89,1.95)}] plot[domain=-2.61:0.43,variable=\t]({1*0.32*cos(\t r)+0*0.32*sin(\t r)},{0*0.32*cos(\t r)+1*0.32*sin(\t r)});
\draw [shift={(0.22,1.87)}] plot[domain=3.48:6.1,variable=\t]({1*0.4*cos(\t r)+0*0.4*sin(\t r)},{0*0.4*cos(\t r)+1*0.4*sin(\t r)});
\draw [shift={(2.8,2.19)}] plot[domain=2.13:4.92,variable=\t]({1*0.43*cos(\t r)+0*0.43*sin(\t r)},{0*0.43*cos(\t r)+1*0.43*sin(\t r)});
\draw [shift={(2.25,3.05)}] plot[domain=3.85:5.29,variable=\t]({1*0.59*cos(\t r)+0*0.59*sin(\t r)},{0*0.59*cos(\t r)+1*0.59*sin(\t r)});
\draw [shift={(1.38,2.5)}] plot[domain=-2.03:0.34,variable=\t]({1*0.46*cos(\t r)+0*0.46*sin(\t r)},{0*0.46*cos(\t r)+1*0.46*sin(\t r)});
\draw [shift={(-0.46,2.04)}] plot[domain=3.13:5.48,variable=\t]({1*0.43*cos(\t r)+0*0.43*sin(\t r)},{0*0.43*cos(\t r)+1*0.43*sin(\t r)});
\draw [shift={(-2.06,-1.68)}] plot[domain=-0.07:2.26,variable=\t]({1*0.37*cos(\t r)+0*0.37*sin(\t r)},{0*0.37*cos(\t r)+1*0.37*sin(\t r)});
\draw [shift={(2.93,-0.79)}] plot[domain=1.74:3.18,variable=\t]({1*0.37*cos(\t r)+0*0.37*sin(\t r)},{0*0.37*cos(\t r)+1*0.37*sin(\t r)});
\draw [shift={(2.22,-0.81)}] plot[domain=0:2.34,variable=\t]({1*0.33*cos(\t r)+0*0.33*sin(\t r)},{0*0.33*cos(\t r)+1*0.33*sin(\t r)});
\begin{scriptsize}
\draw[color=ffqqtt] (-0.12,0.5) node {{\Large$\Omega$}};
\draw[color=black] (-1.01,2.3) node {{\Large$\partial\Omega_k$}};
\end{scriptsize}
\end{tikzpicture}
\caption{\label{Fig 3}}
    \end{figure}
   We will now generalize, via a limit procedure, the previous construction to arbitrary convex domains, or more generally
    domains satisfying a uniform ball condition. Before doing that, let us give a suitable definition for this last condition:
   \begin{definition}
   We say that a bounded domain $\Omega\subset\R^d$ satisfies an exterior ball condition of radius $r>0$ if for every point $x\in\partial\Omega$ there exists $y\in\R^d\setminus\Omega$ such that $|x-y|=r$ and $B(y,r)\cap\Omega=\emptyset$. This is equivalent to say that $\overline\Omega$ coincides with the union of all the bounded connected components of $\R^d\setminus \bigcup_{x\in K}B(x,r)$, where $K=\{x\in\R^d\,:\,d(x,\Omega)=r\}$.
   \end{definition}
   
   As a consequence of this definition, it is easy to see that for every $\Omega\subset\R^d$ satisfying an exterior ball condition of radius $r>0$ there exists a sequence of round polyhedra $\Omega_k$ such that 
   \begin{itemize}
   \item $\Omega\subset\Omega_k$,
   \item $\diam(\Omega_k)\leq \diam(\Omega)+4r$, 
   \item $\partial\Omega_k$ is made of parts of spheres of radius $r$,
   \item $\partial\Omega_k\to\partial\Omega$ in the Hausdorff sense, and $P_{\partial\Omega_k}(x)\to P_{\partial\Omega}(x)$ for a.e. $x\in\Omega$.
   \end{itemize}
   
   We can now state the following
   \begin{proposition}\label{existsf-}
   Suppose that $\Omega\subset\R^d$ is a compact domain satisfying a uniform exterior ball condition of radius $r>0$. Then there exists a larger domain $\widetilde\Omega$ and a constant $C$ only depending on $d,r$ and $L:=\diam(\Omega)$ such that for every positive measure $f^+\ll\lcal^d$ there exists $f^-\ll\lcal^d$, supported on $\widetilde\Omega\setminus \Omega$ with
    $$(\sigma(f^+,f^-))\res\Omega=\sigma(f^+,(P_{\partial\Omega})_\#f^+)$$
and, for every $p\in [1,+\infty]$,
$$ || f^-||_{L^p(\widetilde\Omega)}\leq C || f^+ ||_{L^p(\Omega)}$$
\end{proposition}
\begin{proof}
It is enough to act by approximation. In the case where $\Omega$ is convex, we can write it as an intersection of half-spaces, and hence we can approximate $\Omega$ as the limit of a sequence of polyhedra $\Omega_k$, while in the case where $\Omega$ satisfies a uniform ball condition, we will write it as a limit of round polyhedra (see Figure \ref{Fig 3}) as we pointed out above.

Then, we just build the reflections maps $R_k$ (or $T_k$) as in Propositions \ref{convcase} and \ref{PolyhedraL}, and we get a sequence of measures $f^-_k$ supported on $\widetilde\Omega\setminus\Omega_k$ with $ || f^-_k||_{L^p(\widetilde\Omega)}\leq C || f^+||_{L^p(\Omega_k)}=C|| f^+||_{L^p(\Omega)}$. We also have 
  $$(\sigma(f^+,f^-_k))\res\Omega=\sigma(f^+,(P_{\partial\Omega_k})_\#f^+)\res\Omega.$$
  Then, it is enough to extract a converging subsequence from the sequence $f^-_k$, note that we have $(P_{\partial\Omega_k})_\#f^+\deb(P_{\partial\Omega})_\#f^+$, and use Proposition \ref{stabsigma}.
 \end{proof}
As a consequence, we can now obtain
\begin{theorem} Suppose that $\Omega$ satisfies a uniform exterior ball condition of radius $r>0$. Then, the transport density $\sigma$ between $f^+$ and  $(P_{\partial\Omega})_{\#}f^+$ is in $L^p(\Omega)$ provided $f^+$ is in $L^p(\Omega)$, and 
$$|| \sigma ||_{L^p(\Omega)}\leq C || f^+||_{L^p(\Omega)},$$
where the constant $C$ only depends on $d,r$ and $L=\diam(\Omega)$.
\end{theorem}
\begin{proof} We just need to use Proposition \ref{existsf-}, which guarantees that $\sigma$ is the restriction to $\Omega$ of the transport density between two $L^p$ measures.  
\end{proof}

We finish this section by two remarks on the proof of the above result.
\begin{remark}\label{reminfty}
In this particular case where the transport has not a fixed target measure on $\partial\Omega$, the transport density $\sigma$ linearly depends on $f^+$: in this case, $L^p$ estimates could be obtained via interpolation (via the celebrated Marcinkiewicz interpolation theorem, \cite{Mar,Zyg}) as soon as one has $L^1$ and $L^\infty$ estimates. Since $L^1$ (and $L^p$ for $p<d'$) are well-known, this means that it would be enough to write $L^\infty$ estimates. Yet, we did not see any significant simplification in concentrating on $L^\infty$ estimates instead of $L^p$, which is the reason why we decided not to evoke general interpolation theorems but we performed explicit estimates. In the same way, one could get also a $L^{p,q}$ estimates (see \cite{Lor,11} for the definition of the $L^{p,q}$ spaces).\end{remark} 

\begin{remark}
Another observation concerns the fact that we proved Proposition \ref{existsf-} by approximation. Apart from the fact that we first developed the convex case (just for the sake of simplicity), the reader could have preferred a direct formulation, valid in the case of an arbitrary domain $\Omega$ with an exterior ball condition, instead of passing through round polyhedra. This would be possible, by defining a map $T(x):=P_{\partial\Omega}(x)+c(P_{\partial\Omega}(x)-x)$, for small $c>0$. It can be proven, by studying the properties of the Jacobian of $P_{\partial\Omega}$, that $T$ is injective and $|\det(DT)|$ is bounded from below as soon as $c$ is small (depending on $L$ and $r$), but we considered that the proof in the case of round polyhedra was easier.
\end{remark}

 \section{An $L^\infty$ bound on $f^-$ with respect to $\haus^{d-1}\res(\partial\Omega)$ is not enough} \label{6}

 In this section we show that the $L^\infty$ estimates for the transport density (again, note by Remark \ref{reminfty} that the case $p=\infty$ is the most interesting one) fail if we only assume summability (or boundedness) of the densities of $f^+$ w.r.t. the Lebesgue measure on $\Omega$ and of $f^-$ w.r.t. the Hausdorff measure $\haus^{d-1}$ on $\partial\Omega$. Indeed, when we consider a domain $\Omega$ with a uniform exterior ball condition and we take $f^+\in L^\infty$, we can easily prove that $(P_{\partial\Omega})_\#f^+$ has a bounded density w.r.t.  $\haus^{d-1}\res\partial\Omega$. One could wonder whether this is the correct assumption to prove, for instance, $\sigma\in L^\infty$, and the answer is not.
 
 We will construct an example of $f^\pm$, where $f^+$ has a bounded density w.r.t. $\lcal^d$ in $\Omega$ and $f^-$ w.r.t.  $\haus^{d-1}\res\partial\Omega$ (for instance $\Omega$ is a big square containing the support of $f^+$ and its boundary contains the support of $f^-$), but $\sigma\notin L^\infty$ (we will also investigate the summability of $\sigma$).
  Set 
  $$f^+:=\lcal^2\res A,\quad f^-:=\mathcal{H}^1\res({[2,3]\times \{0\}})$$
 where $A$ is a trapeze with vertices $(0,0),\;(1,0),\;(1,\frac{4}{5})$ and $(0,\frac{6}{5} )$ (see Figure \ref{Fig 4}).
%{\color{red} centrer la figure. n'utiliser les couleurs que si nÃÂÃÂÃÂÃÂ©cessaire}

 For every $\varepsilon \in [0,1]$, let $ l_{\varepsilon}$ be the segment joining the two points $(0,w(\varepsilon))$ and $(2+\varepsilon,0)$, where
 $$w(\ve):=\frac{2 \varepsilon(2+\varepsilon)}{ 3+2\varepsilon}.$$
 First, it is easy to see that $f^+(\Delta_{\varepsilon})=f^-(\Delta_{\varepsilon})$ for every $\varepsilon \in [0,1]$ where $\Delta_{\varepsilon}$ is the triangle limited by $(0,0),\;(2+\varepsilon,0)$ and $(0,w(\ve) ).$  
 Then by \cite{8}, we can construct an optimal mapping $T$, which pushes $f^+$ to $f^-,$ with $\{ l_{\varepsilon} \}$ as its transfer rays.\\ 

 Let $\sigma$ be the transport density between $f^+$ and $f^-$.  For simplicity of notation, we denote the ball of center $(2,0)$ and radius $r$ by $ B_r$, and in this case we have 
 $$\sigma(B_{2r})=\int
 %_{\Delta_{1} \times \Delta_{1}} 
 \mathcal{H}^1(B_{2r} \cap [x,y])\mathrm{d}\gamma(x,y) \geq r\gamma\left( \{(x,y): B_{r} \cap [x,y] \neq \emptyset\}\right),$$
 where we used the fact that for every $(x,y) $ s.t. $B_{r} \cap [x,y] \neq \emptyset$, we have $ \mathcal{H}^1(B_{2r} \cap [x,y]) \geq r $.

Then, note that there exists a value $\varepsilon_r \in (0,1)$ such that $\{x: B_{r} \cap [x,T(x)] \neq \emptyset\}=\Delta_{\varepsilon_r}$ and $l_{\varepsilon_r}$ is tangent to the ball $B_r$.
  \begin{figure}
\begin{tikzpicture}[line cap=round,line join=round,>=triangle 45,x=3.0cm,y=3.0cm]
\clip(-0.24,-0.35) rectangle (3.05,1.71);
\draw [color=ffqqqq] (0,1.2)-- (0,0);
\draw [color=ffqqqq] (0,0)-- (1,0);
\draw [color=ffqqqq] (1,0)-- (1,0.8);
\draw [color=ffqqqq] (1,0.8)-- (0,1.2);
\draw [color=ffwwqq] (2,0)-- (3,0);
\draw [shift={(2,0)},color=blue]  plot[domain=0:3.148,variable=\t]({1*0.15*cos(\t r)+0*0.15*sin(\t r)},{0*0.15*cos(\t r)+1*0.15*sin(\t r)});
\draw [shift={(2,0)},color=blue]  plot[domain=0:3.148,variable=\t]({1*0.3*cos(\t r)+0*0.3*sin(\t r)},{0*0.3*cos(\t r)+1*0.3*sin(\t r)});
\draw (0,0.73)-- (2.56,0);
\draw (2.3,0)-- (0,0.51);
\draw (0,0.32)-- (2.15,0);
\draw (0,0.14)-- (2.04,0);
\begin{scriptsize}
\draw[color=ffqqqq] (0.58,1.1);
\draw[color=ffwwqq] (2.51,-0.05);
\fill [color=ttfftt] (0,0.73) circle (1.5pt);
\fill [color=ttfftt] (0.48,0.59) circle (1.5pt);
\draw[color=black] (0.5,0.475) node {{\Large$x$}};
\fill [color=ttfftt] (1,0.44) circle (1.5pt);
\fill [color=ttfftt] (0,0.51) circle (1.5pt);
\fill [color=ttfftt] (0.48,0.41) circle (1.5pt);
\fill [color=ttfftt] (1,0.29) circle (1.5pt);
\fill [color=ttfftt] (0,0.32) circle (1.5pt);
\fill [color=ttfftt] (0.48,0.25) circle (1.5pt);
\fill [color=ttfftt] (1,0.17) circle (1.5pt);
\fill [color=ttfftt] (1,0.07) circle (1.5pt);
\fill [color=ttfftt] (0.48,0.11) circle (1.5pt);
\fill [color=ttfftt] (0,0.14) circle (1.5pt);
\fill [color=ffffqq] (2.56,0) circle (1.5pt);
\draw[color=red,very thick] (2.04,0)--(3,0);
\draw[color=black] (2.3,-0.1) node {{\Large$T(x)$}};
\fill [color=ffffqq] (2.3,0) circle (1.5pt);
\fill [color=ffffqq] (2.15,0) circle (1.5pt);
\fill [color=ffffqq] (2.04,0) circle (1.5pt);
\draw[color=ffqqqq] (0.4,0.9) node {{\Large$f^+$}};
\draw[color=ffwwqq] (3,-0.125) node {{\Large$f^-$}};
\draw[color=black] (2,0.4) node {{\Large$B_{2 r}$}};
\end{scriptsize}
\end{tikzpicture}
\caption{\label{Fig 4}}
 \end{figure}
 Then 
 $$\sigma(B_{2r}) \geq r f^+(\{x:\; B_{r} \cap [x,T(x)] \neq \emptyset\})=r f^+(\Delta_{\varepsilon_r}) \simeq r \varepsilon_r.$$
 If we denote by $\theta$ the angle between the two segments $[(0,0),(2+\varepsilon_r,0)]$ and $l_{\varepsilon_r}$, then we have 
 $$ \sin(\theta)=\frac{r}{\varepsilon_r}$$ and $ \tan(\theta)\simeq \varepsilon_r \mbox{ for } r \;\mbox{small enough}.$
  
 But for $r$ small enough, $\theta\simeq 0$ and we get $ \varepsilon_r \simeq r^{\frac{1}{2}}$. Thus, for $r$ small enough 
\begin{equation}\label{sig32}
\sigma(B_{2r}) \geq C r^{\frac{3}{2}},
\end{equation}
 which implies that $\sigma$ cannot be bounded in a neighborhood of $(2,0)$, otherwise we would have 
 $$  c r^2=|| \sigma ||_{L^{\infty}(\Delta_1)}|B_{2r}| \geq \sigma(B_{2r}) \geq C r^{\frac{3}{2}}$$
 which is a contradiction for small $r$.
 In addition, it is possible to see $\sigma \notin L^{4}(\Delta_1)$, otherwise, by H\"older inequality, we would get
 $$\frac{\sigma(B_{2r})}{r^{3/2}}\leq \frac{|B_{2r}|^{3/4}\left(\int_{B_{2r}}\sigma^4\right)^{1/4}}{r^{3/2}}\to 0,$$
% $$ cr^{\frac {2}{ p^\prime}}=|| \sigma ||_{L^{p}(B_{2r})}|B_{2r}|^{\frac{1}{p^\prime}} \geq \sigma(B_{2r}) \geq C r^{\frac{3}{2}}$$
 which is a contradiction with \eqref{sig32}. 
 
 Actually, a finer analysis even proves $\sigma \in L^{p}(\Delta_1)$ if and only if $p<3$.
 To prove this we need to use heavier computations. Fix $\varepsilon_0$ small enough and take $x \in \Delta_{\varepsilon_0}$: there exist $\varepsilon \in [0,\varepsilon_0]$ and $s \in [0,1]$ such that $$x=(1-s)(2+\varepsilon,0) + s(0,w(\varepsilon)).$$
 For all $\varphi \in C(\Delta_{\varepsilon_0})$ we have 
 $$< \sigma,\varphi>:=\int_{0}^{1}\int_{\Delta_{\varepsilon_0}}|x-T(x)|\varphi((1-t)x+t T(x))f^+(x)\mathrm{d}x\mathrm{d}t$$
and by a change of variable, we get, in the variable $(\ve,s)$,
%$x=(1-s)(2+\varepsilon,0) + s(0,w(\varepsilon)):=U(\varepsilon,s)$ allows to compute the value of $\sigma$ in the $(s,\ve)$ variables:
%$$< \sigma,\varphi>$$
%$$=\int_{0}^{\varepsilon_0}\int_{0}^{1}\int_{s}^{1} \sqrt{(2+\varepsilon)^2 + w(\varepsilon)^2}\varphi((1-s)(2+\varepsilon),s w(\varepsilon)) f^+((1-t)(2+\varepsilon),t w(\varepsilon))|\mbox{J}(\varepsilon,t)|\mathrm{d}t\mathrm{d}s\mathrm{d}\varepsilon$$ \\
%where $|\mbox{J}|:=|D_{(\varepsilon,s)}(x_1,x_2)|.$ \\ \\
 \begin{eqnarray*}
 \sigma(\varepsilon,s)&=&\frac{\sqrt{(2+\varepsilon)^2 + w(\varepsilon)^2}\int_{s}^{1} f^+((1-t)(2+\varepsilon),t w(\varepsilon))|J(t,\varepsilon)|\mathrm{d}t}{|J(s,\varepsilon)|}\\
& =&\frac{\sqrt{(2+\varepsilon)^2 + w(\varepsilon)^2}\int_{1-\frac{1}{2+\varepsilon}}^{1}|J(\varepsilon,t)|\mathrm{d}t}{|J(\varepsilon,s)|} \simeq \frac{\int_{1-\frac{1}{2+\varepsilon}}^{1}|J(\varepsilon,t)|\mathrm{d}t}{|J(\varepsilon,s)|},
\end{eqnarray*}
where $|J|:=|D_{(\varepsilon,s)}(x_1,x_2)|$. Using $|J(\varepsilon,s)| \simeq s + \varepsilon $
we get $$\sigma(\varepsilon,s) \simeq \frac{1}{s+\varepsilon}$$
 and $$|| \sigma ||_{L^{p}(\Delta_1)}^p\; \simeq \int_{0}^{\varepsilon_0}\int_{0}^{\varepsilon_0}\frac{1}{(s+\varepsilon)^{p-1}}\mathrm{d}s\mathrm{d}\varepsilon \simeq \int_{0}^{\varepsilon_0}\frac{1}{\varepsilon^{p-2}}\mathrm{d}\varepsilon.  $$ 
 Notice that as $d=2$ and $f^+ \in  L^{\infty}(\Delta_1)$, by \cite{San09} %{\color{red} (tu es sÃÂÃÂÃÂÃÂ»r que ce rÃÂÃÂÃÂÃÂ©sultat ÃÂÃÂÃÂÃÂ©tait prouvÃÂÃÂÃÂÃÂ© pour la premiÃÂÃÂÃÂÃÂ¨re fois dans mon papier ? il ne faut pas donner l'impression que c'ÃÂÃÂÃÂÃÂ©tait MON idÃÂÃÂÃÂÃÂ©e)} 
we know that automatically $\sigma \in 
 L^{p}(\Delta_1)$ for all $p<2$. The fact that here we get $\sigma \in L^{p}(\Delta_1)$ for all $p<3$ depends on the fact that we send a mass $f^+$ to a mass $f^-$ which is distributed on a segment, and not to a Dirac mass. In some sense, we are not in the worst possible case! 
\end{document}